\documentclass[11pt]{amsart}
\usepackage{amsmath, amsthm, amssymb}

\newtheorem{theorem}{Theorem}

\newtheorem{lemma}{Lemma}

\def\L{L}
\def\R{{\mathbb R}}
\def\eps{\varepsilon}
\def\vrrr#1#2{\vrule height#1mm depth#2mm}
\def\vrr{\vrule height6mm depth4mm}
\def\celll#1#2#3#4{\vbox to6mm{\hbox to #1mm{\hss#4\hss\vrrr{#2}{#3}}\hrule\vss}}
\def\cel#1#2{\vbox to6mm{\hbox to #1mm{\hss#2\hss\vrr}\hrule\vss}}

\title[Uniqueness]{Uniqueness of singular self-similar solutions \\
     of a semilinear parabolic equation}
\author{Pavol Quittner}
\address{Department of Applied Mathematics and Statistics, Comenius University 
         Mlynsk\'a dolina, 84248 Bratislava, Slovakia}
\email{quittner@fmph.uniba.sk}

\begin{document}

\begin{abstract}
We study the uniqueness of singular radial (forward and backward)
self-similar positive solutions of the equation
$
u_t-\Delta u = u^p, \quad x\in\R^n,\ t>0,
$
where $p\geq(n+2)/(n-2)_+$.
\end{abstract}

\maketitle
\section{Introduction} \label{sec-intro}

In this paper we study the uniqueness of singular radial (forward and backward) 
self-similar positive solutions of the equation
\begin{equation} \label{Fuj}
u_t-\Delta u = u^p, \qquad x\in\R^n,\ t>0, 
\end{equation}
where $p\geq p_S:=(n+2)/(n-2)_+$.
More precisely, we are interested in unbounded positive self-similar
solutions of the equation
\begin{equation} \label{Fujsr}
U_t-U_{rr}-\frac{n-1}r U_r = U^p \qquad \hbox{for }r,t>0. 
\end{equation}
Such solutions are of the form 
$$U(r,t)=t^{-1/(p-1)}w(r/\sqrt{t})\quad\hbox{or}\quad
U(r,t)=(T-t)^{-1/(p-1)}w(r/\sqrt{T-t}),$$
where $w=w(r)$ is an unbounded positive solution of the equation
\begin{equation} \label{eqw}
 w''+\bigl(\frac{n-1}r+\frac r2\Bigr)w'+\frac1{p-1}w+w^p=0
\qquad \hbox{for }r>0,
\end{equation}
or
\begin{equation} \label{eqwb}
 w''+\bigl(\frac{n-1}r-\frac r2\Bigr)w'-\frac1{p-1}w+w^p=0
\qquad \hbox{for }r>0,
\end{equation}
respectively.
Notice that stationary solutions $U=U(r)$ of \eqref{Fujsr}
solve the equation 
\begin{equation} \label{eqv}
U_{rr}+\frac{n-1}r U_r+U^p=0 \qquad \hbox{for }r>0,
\end{equation}
and equations \eqref{eqw} and \eqref{eqwb}
can be seen as perturbations of \eqref{eqv}.

Assume $p>1$.
It is well known that \eqref{eqv} possesses (a continuum of)
bounded positive solutions if and only if $p\geq p_S$,
see \cite{Fow31,Wang93} or \cite{SPP}.
If $p>p_{sg}:=n/(n-2)_+$ then \eqref{eqv}
possesses the unbounded positive solution 
\begin{equation} \label{Ustar}
U_*(r):={\L}r^{-2/(p-1)}, \quad\hbox{where}\quad
{\L}^{p-1}:=\frac2{(p-1)^2}\bigl((n-2)p-n\bigr),
\end{equation}
and this is the only unbounded positive solution
if $p>p_S$, see \cite{SZ94}. On the other hand, if $p_{sg}<p\leq p_S$
then \eqref{eqv} possesses a continuum of unbounded
positive solutions, see \cite[Proposition 2.2]{CL99} and \cite{SZ94}.
These results can also be easily derived 
by using the transformation 
$$v(s)=r^{2/(p-1)}U(r), \qquad s=\log r, $$
and the corresponding phase plane analysis for $(v,v')$
(see \cite[Section 9]{SPP} if $p\geq p_S$).
The function $v$ solves the equation
$$ v''+\beta v'+v^p-\gamma v=0, \qquad s\in\R$$
where 
$$ \beta:=\frac1{p-1}((n-2)p-(n+2)),\qquad 
\gamma=\frac2{(p-1)^2}((n-2)p-n).  $$
Notice that $\gamma=L^{p-1}>0$ if $p>p_{sg}$,
while $\beta>0$ if and only if $p>p_S$.
If $p>p_S$ then bounded positive solutions of \eqref{eqv}
correspond to the trajectories in the phase plane
joining the equilibria $(0,0)$ and $(L,0)$, 
and the unbounded solution $U_*$ corresponds to the equilibrium $(L,0)$,
see \cite[Figures 5--7]{SPP}.
If $p_{sg}<p<p_S$ then  
one can use the transformation $s\mapsto -s$
(which changes $\beta$ to $-\beta$) to obtain similar
pictures as for $p>p_S$, but now the trajectories 
joining $(0,0)$ and $(L,0)$ correspond to unbounded positive 
solutions of \eqref{eqv}.
This shows a kind of duality between bounded and unbounded
solutions and the cases $p>p_S$ and $p_{sg}<p<p_S$.
This duality concerns not only the existence but
also the qualitative properties of solutions.
For example, if we set
$$ p_{JL} :=\begin{cases}\ +\infty & \hbox{ if }n\leq10, \\
        1+\frac4{n-4-2\sqrt{n-1}} & \hbox{ if }n>10,
          \end{cases}
\quad
p_{JL}^* :=\hbox{$1+\frac4{n-4+2\sqrt{n-1}}$} \ \ \hbox{ if }n>2, \\
$$
then the bounded solutions of \eqref{eqv} in the range $p>p_S$
intersect each other (and the singular solution $U_*$)
if and only if $p<p_{JL}$, and, similarly, the unbounded
solutions of \eqref{eqv} in the range $p\in(p_{sg},p_S)$ 
intersect each other (and the singular solution $U_*$)
if and only if $p>p_{JL}^*$.
(Notice that $p_{JL}>p_S$ and $p_{JL}^*\in(p_{sg},p_S)$ if $n>2$.)

Some of the features of \eqref{eqv} mentioned above
can also be expected for its perturbations \eqref{eqw} and \eqref{eqwb};
however a complete description of all bounded and unbounded
solutions of these equations is still missing
and their analysis is much more difficult.
Since these solutions 
play an important role in the study of the large-time
or blow-up behavior of solutions of the model problem \eqref{Fuj},
the list of available results and references is very long
and we recall just a few of them.

First let us mention that $U_*$ is a positive unbounded solution
of both \eqref{eqw} and \eqref{eqwb} if $p>p_{sg}$.
If $n>10$ and $p>p_{JL}$ then the linearization 
of \eqref{eqwb} at $U_*$ (in a suitable function space)
defines a self-adjoint operator $A_*$ with spectrum $\sigma(A_*)$
consisting of a countable sequence of eigenvalues (see \cite{HV94}),
and we set
$$p_L:=(n-4)/(n-10)=\sup\{p>p_{JL}: 0\in\sigma(A_*)\},$$
cf.~\cite{Mi09}. 
The couple $(p_L,U_*)$ is a bifurcation point for
positive solutions of \eqref{eqwb},
cf.~\cite[Remark 2.5]{Mi10} and the arguments in \cite{FM}.  

Equation \eqref{eqwb} has a unique bounded positive solution
$w\equiv (p-1)^{-1/(p-1)}$ if $p\leq p_S$
(see \cite{GK85}),
it has an infinite sequence of bounded positive solutions 
if $p_S<p<p_{JL}$ (see \cite{Lep88,BQ89}), 
it has at least one bounded positive non-constant solution
if $p_{JL}<p<p_{L}$ (see \cite{Lep90}), 
and it does not have bounded positive
non-constant solutions if $p\geq p_L$ 
(see \cite{Mi09,Mi10}). 
Other interesting properties of bounded positive solutions
of \eqref{eqwb} can be found in \cite{Matos,FM,FP}, for example.
The uniqueness of the unbounded positive solution $U_*$ 
of \eqref{eqwb} for $p>p_S$ 
has been proved in \cite[Theorem~1.2]{Mi10};
see also previous results 
in \cite[Lemma~4.9]{FMP} and \cite[Proposition~A.1]{MM04}. 
On the other hand, this uniqueness fails if $p\in(p_{sg},p_S)$,
see \cite{SY11}. If $p=p_S$ then a
nonexistence result for unbounded positive solutions
of \eqref{eqwb} belonging to an energy space
was obtained in \cite[Lemma~5.1]{Suz08}.  

Equation \eqref{eqw} possesses a continuum of bounded positive
solutions for all $p>p_F:=1+2/n$, but the exponents
$p_S$ and $p_{JL}$ are again critical in some sense,
see \cite{HW82,PTW86,Yan96,DH98,SW03,Nai06,Nai08,Nai12} and the references therein.
In particular, the set of such solutions is bounded in $L^\infty$
if and only if $p<p_S$.
Fix $p>p_F$.
Given a positive solution $w$ of \eqref{eqw},
the limit $\ell(w):=\lim_{r\to\infty} w(r)r^{2/(p-1)}$ exists and 
$$L^*:=\sup\{\ell(w):
 \hbox{$w$ is a bounded positive solution of \eqref{eqw}}\}\in(0,\infty).$$
The only result concerning unbounded positive solutions
(that we are aware of)
shows that if $p_F<p<p_{JL}$ then any unbounded positive solution
of \eqref{eqw} satisfies $\ell(w)\leq L^*$, see
\cite[Lemma 7.1]{Nai12}.

Let us emphasize that all the results on unbounded positive solutions
of \eqref{eqw} and \eqref{eqwb} mentioned above
were motivated by (and immediately applied in) 
the study of solutions \eqref{Fuj}. 
This is also the case of the following uniqueness theorem
which is the main result of this paper and which plays
an important role in the study of threshold solutions of \eqref{Fuj}
in \cite{Qthr}.

\begin{theorem} \label{thm-uniq}
Let $p\geq p_S$ 
and let $w$ be a positive unbounded solution of \eqref{eqw} or \eqref{eqwb}.
If $p=p_S$ assume also that 
the number of sign changes of $w-U_*$ is finite.
Then $w=U_*$.
\end{theorem}

Our proof of Theorem~\ref{thm-uniq} is a modification of the proof
of an analogous result for equation \eqref{eqv} in \cite{SZ94};
this modification is far from straightforward 
in the critical case $p=p_S$.

\section{Proof of Theorem~\ref{thm-uniq}} \label{sec4}

\begin{lemma} \label{lem-ub}
Let $n>2$, $p>1$, $\alpha:=2/(p-1)$, and
let $w$ be a positive solution of \eqref{eqw}.
Then $w'\leq0$  and  there exists $C>0$ such that
\begin{equation} \label{boundw}
r^{\alpha+i}|w^{(i)}(r)|\leq C,\quad r\in(0,1),\quad i=0,1,2, 
\end{equation}
where $w^{(0)}=w$, $w^{(1)}=w'$ and $w^{(2)}=w''$.
If $w$ is bounded then $w'(r)=O(r)$ as $r\to0$.
\end{lemma}

\begin{proof}
Assume on the contrary $w'(r_0)>0$ for some $r_0>0$. 
Then \eqref{eqw} guarantees $w''(r)<0$ and $w'(r)>w'(r_0)$ for
$r<r_0$, and the inequality $(r^{n-1}w'(r))'\leq0$ for $r<r_0$
guarantees $w'(r)\geq c_0r^{1-n}$ for $r<r_0$
which contradicts
the estimate $\int_\rho^{r_0}w'(r)\,dr\leq w(r_0)$ for $\rho$
small enough.

Estimates \eqref{boundw} follow from
the scaling and doubling arguments as in \cite{PQS1}.
In fact, assume on the contrary that the function
$$ M(r):=w(r)^{(p-1)/2}+|w'(r)|^{(p-1)/(p+1)}+|w''(r)|r^{(p-1)/(2p)}$$
satisfies $M_k:=M(r_k)>2k/r_k$ for some $r_k\to0$.
Then \cite[Lemma~5.1]{PQS1} guarantees that we may assume
$M(r)\leq 2M(r_k)$ whenever $|r-r_k|<k/M_k$.
Set
$v_k(\rho):=\lambda_k^{2/(p-1)}w(r_k+\lambda_k\rho)$
where $\lambda_k=1/M_k$. Then 
$v_k,v_k',v_k''$ are uniformly bounded on $(-k,k)$
and $v_k$ solve the equation
$$ v_k''+\Bigl(\frac{n-1}{r_k/\lambda_k+\rho}
   +\frac12\lambda_k(r_k+\lambda_k\rho)\Bigr)v_k'
   +\frac1{p-1}\lambda_k^2v_k+v_k^p=0, \quad \rho\in(-k,k), $$
where $r_k/\lambda_k=r_kM_k>2k$.
Consequently, a suitable subsequence of
$\{v_k\}$ converges locally uniformly
to a positive solution $v$ of the equation $v''+v^p=0$ in $\R$
which yields a contradiction with the corresponding
Liouville theorem (see \cite[Theorem~8.1]{SPP}, for example).

Alternatively, one can also modify the arguments
in the proofs of \cite[Theorem 2.1, Lemma 2.1 and the subsequent Remark]{NS86} 
to prove \eqref{boundw}.

Next assume that $w$ is bounded.
Then there exist $r_k\to0$ such that $r_k^{n-1}w'(r_k)\to0$.
Consequently, 
integrating the inequality
$$ -(r^{n-1}w'(r))'\leq r^{n-1}\Bigl(\frac1{p-1}w+w^p\Bigr)=O(r^{n-1}) $$
from $r_k$ to $r$ and passing to the limit yields
$ -r^{n-1}w'(r) = O(r^n)$, 
hence $w'(r)=O(r)$.
\end{proof}

\begin{lemma} \label{lem-v0}
Let $p\geq p_S$, let $w$ be a positive solution of \eqref{eqw}
and $v(r):=w(r)r^{\alpha}$, where $\alpha:=2/(p-1)$.
If $v(r)\to0$ and $rv'(r)\to0$ as $r\to0$
then $w$ is bounded.
\end{lemma}

\begin{proof}
The proof is a modification of the proof of 
\cite[Lemma 2.4]{SZ94}.

Lemma~\ref{lem-ub}
guarantees that $w$ is bounded on $[1,\infty)$
and $v$ satisfies
\begin{equation} \label{estv}
 v(r)+|v'(r)r|+|v''(r)r^2|\leq C \quad\hbox{for }r\in(0,1).
\end{equation}
Notice also that $v$ is a solution of
\begin{equation} \label{eqv1}
v''+\Bigl(\frac{n-1-2\alpha}r+\frac r2\Bigr)v'+\frac1{r^2}(v^p-{\L}^{p-1}v)=0,
\qquad r>0,
\end{equation}
where ${\L}$ is defined in \eqref{Ustar}.
Consider $\eps>0$ small and set
$$ \begin{aligned}
 a_1^\pm &:=\alpha-\frac12(n-2-\sqrt{(n-2)^2\mp4\eps})=\alpha+O(\eps),\\
 a_2^\pm &:=\alpha-\frac12(n-2+\sqrt{(n-2)^2\mp4\eps})=\alpha+2-n+O(\eps)<0.
\end{aligned}$$
Choose $r_\eps>0$ small such that
$v^{p-1}(r)+(1-a_2^-/2)r^2<\eps$ for $r<r_\eps$.
Then \eqref{eqv1} guarantees
\begin{equation} \label{ineqv1}
\frac r2 v'+\frac{2-a_2^+}2v
+v''+\frac{n-1-2\alpha}r v'-\frac{({\L}^{p-1}-\eps)}{r^2} v\geq0,
\qquad 0<r<r_\eps,
\end{equation}
and
\begin{equation} \label{ineqv2}
\frac r2 v'+\frac{2-a_2^-}2v    
+v''+\frac{n-1-2\alpha}r v'-\frac{({\L}^{p-1}+\eps)}{r^2} v\leq0,
\qquad 0<r<r_\eps. 
\end{equation}
Inequality \eqref{ineqv1} can be written as
$$ \frac r2 v'+\frac{2-a_2^+}2v+\Bigl(D-\frac{a_2^+-1}r\Bigr)\Bigl(D-\frac{a_1^+}r\Bigr)v\geq0,$$
where $D=d/dr$.
Multiplying the last inequality by $r^{1-a_2^+}$ we see that the function
$$ \psi(r):=\frac12 r^{2-a_2^+}v+r^{1-a_2^+}\Bigl(D-\frac{a_1^+}r\Bigr)v $$
satisfies $\psi'\geq0$. Since $\psi(r)\to0$ as $r\to0$,
we have $\psi(r)\geq0$, hence
$$ \frac12 rv+\Bigl(D-\frac{a_1^+}r\Bigr)v \geq0, \qquad 0<r<r_\eps.$$
Since $r_\eps^2<\eps$, we also have
$$ \Bigl(D-\frac{\tilde a_1}r\Bigr)v \geq0, \qquad 0<r<r_\eps,$$
where $\tilde a_1:=a_1^+-\eps$.
Integrating from $r$ to $r_\eps$ yields
$$ v(r)\leq c_\eps r^{\tilde a_1}=c_\eps r^{\alpha+O(\eps)},
\qquad  0<r<r_\eps.$$
Similarly, considering \eqref{ineqv2} instead of \eqref{ineqv1},
we obtain
$$  \Bigl(D-\frac{a_1^-}r\Bigr)v \leq0\quad\hbox{and}\quad
v(r)\geq \tilde c_\eps r^{a_1^-}=\tilde c_\eps r^{\alpha+O(\eps)},
\qquad  0<r<r_\eps.$$
Thus $\tilde a_1v(r)\leq rv'(r)\leq a_1^- v(r)$ for $0<r<r_\eps$,
and
\begin{equation} \label{eqv2}
 v''+\frac{n-1-2\alpha}r v'-\frac{{\L}^{p-1}}{r^2}v=g(r)
\end{equation}
with
$$ |g(r)|=\Big|\frac{v^p}{r^2}+\frac r2 v'\Big|\leq Cr^{\alpha+O(\eps)}.$$
Now the representation of solutions of \eqref{eqv2} 
implies that $v$ is bounded by $Cr^\alpha$, hence $w$ is bounded:
\end{proof}

\begin{lemma} \label{lem-rdhr}
Let $p\geq p_S$, let $w$ be a positive solution of \eqref{eqw},
$v(r)=r^\alpha w(r)$, where $\alpha=2/(p-1)$.
If $p=p_S$ assume also $z(v-{\L})<\infty$.

Then
\begin{equation} \label{rdhr}
rv'(r)\to0\ \hbox{ as }\ r\to0.
\end{equation}
If $p=p_S$ and $v\not\equiv {\L}$ then $v(r)\to0$ as $r\to0$.
\end{lemma}

\begin{proof}
Notice that 
$v$ is a solution of \eqref{eqv1} satisfying \eqref{estv}.
Set $h:=v/{\L}$.  

First assume $p=p_S$.
Then $h$ is a positive solution of
\begin{equation} \label{eqh}
h''+\Bigl(\frac1r+\frac r2\Bigr)h'+\frac{\alpha^2}{r^2}(h^p-h)=0,\qquad r>0,
\end{equation}
and satisfies
\begin{equation} \label{esth}
 h(r)+r|h'(r)|+r^2|h''(r)|\leq C,\qquad r\in(0,1).
\end{equation}
Assume $h\not\equiv1$, hence $h'\not\equiv0$.
Multiplying \eqref{eqh} by $r^2$ we obtain
\begin{equation} \label{rrh}
r(rh')'+\frac{r^3}2h'+d(h(r))=0,
\quad\hbox{where }\ d(\xi):=\alpha^2(\xi^p-\xi).
\end{equation} 
Multiplying \eqref{rrh} by $h'$ and
integrating from $\rho$ to $R$ we obtain
\begin{equation} \label{cRrho}
 c(R)-c(\rho)=-\frac12\int_\rho^Rh'(r)^2r^3\,dr<0,
\end{equation}
where
\begin{equation} \label{cb}
c(r):=\frac{r^2}2 h'(r)^2+b(h(r)),
\qquad b(\xi):=\alpha^2\Bigl(\frac1{p+1}\xi^{p+1}-\frac12 \xi^2\Bigr).
\end{equation}
Set $c_0:=\lim_{r\to0}c(r)$, $c_\infty:=\lim_{r\to\infty}c(r)$.
Then $c_0>c_\infty$.

We will first prove that
\begin{equation} \label{asymp1}
r(rh'(r))'\to0 \hbox{ as }r\to0.
\end{equation}
In fact, the boundedness of $rh'(r)$ for $r\leq1$ shows 
the existence of $r_k\to0$ such that
$r_k(r_kh'(r_k))'\to0$, hence $d(h(r_k))\to0$
due to \eqref{esth}, \eqref{rrh},
and we may assume that either $h(r_k)\to0$ or $h(r_k)\to1$.
Assume that there exist $\tilde r_k\to0$ such that
the sequence $\eta_k:=\tilde r_k(\tilde r_kh'(\tilde r_k))'$
satisfies
$|\eta_k|\geq2\eps_0>0$.
Then $|d(h(\tilde r_k))|\geq\eps_0$ for $k$ large,
hence $|h(\tilde r_k)-h(r_k)|\geq\delta_0>0$
and, consequently, 
there exist $\hat r_k\to0$ such that $h'(\hat r_k)=0$
and $\liminf h(\hat r_k)\leq 1$.
Hence 
$c_0=\lim c(\hat r_k)=\lim b(h(\hat r_k))\leq0$.

If (for a suitable subsequence) $\eta_k\leq-2\eps_0$
then $d(h(\tilde r_k))\geq\eps_0$,
hence $h(\tilde r_k)>1$ and
$z(h-1)<\infty$ implies $h(r)\geq1$ for $r$ small.
Consequently, $h(\hat r_k)\to1$,
hence $c_0=\lim b(h(\hat r_k))=\min b\leq c_\infty$,
which yields a contradiction.

Consequently, $\eta_k\geq2\eps_0$
and $d(h(\tilde r_k))\leq-\eps_0$,
hence $h(\tilde r_k)<1$ and 
$z(h-1)<\infty$ implies $h(r)\leq1$ for $r$ small.

If $h(r_k)\to1$ then we may assume $h(\hat r_k)\to1$
and we obtain a contradiction as above.

Therefore $h(r_k)\to0$. Since 
$|h(\tilde r_k)-h(r_k)|\geq\delta_0>0$ and $h>0$, we have
$h(\tilde r_k)\geq\delta_0$ for $k$ large enough.
Consequently, we can find $\hat r_k^{(i)}\to0$, $i=1,2$, such that
$h'(r_k^{(i)})=0$, $h(r_k^{(1)})\to0$, $h(r_k^{(2)})\in[\delta_0,1]$.
However, this contradicts the fact that 
$b(h(r_k^{(i)}))\to c_0$, $i=1,2$.
This shows that \eqref{asymp1} is true.

Since \eqref{asymp1} guarantees $d(h(r))\to0$ as $r\to0$,
we have either $h(r)\to0$ or $h(r)\to1$ as $r\to0$.
If $h(r)\to1$ as $r\to0$ then 
$c_0\leq\lim_{r\to0} b(h(r))=\min b$,
which yields a contradiction.
Hence $h(r)\to0$ and $b(h(r))\to0$ as $r\to0$.
Since $h$ is bounded on $(0,1)$, we can find $r_k\to0$ such that
$r_kh'(r_k)\to0$, hence $c_0=\lim b(h(r_k))=0$
and, consequently, 
\eqref{rdhr} is true.

Next assume $p>p_S$. 
Then  \eqref{rdhr} follows by
a straightforward modification of the proofs
of \cite[Lemmas~2.1--2.2]{SZ94}.  
In fact, multiplying \eqref{eqv1} by $v'(r)r^2$ and integrating
one obtains
\begin{equation} \label{eqv1int}
\left\{\quad\begin{aligned}
&\frac{v'^2r^2}2\Big|_\rho^R+(n-2-2\alpha)\int_\rho^Rv'^2r\,dr \\
&\qquad\qquad +\frac12\int_\rho^R v'^2r^3\,dr
 +\Bigl(\frac{v^{p+1}}{p+1}-\frac{{\L}^{p-1}v^2}2\Bigr)\Big|_\rho^R=0,
\end{aligned}\right.
\end{equation}
hence \eqref{estv} guarantees
\begin{equation} \label{estvint}
\int_\rho^Rv'^2r\,dr\leq C\ \hbox{ whenever }1\geq R>\rho>0.
\end{equation}
Assume that there exist $r_k\to0$ such that
$|v'(r_k)r_k|\geq c>0$.
Since $|(v'^2r^2)'|\leq M/r$ for $r\leq1$ due to \eqref{estv},
we have
$$ |v'^2r^2-v'(r_k)^2r_k^2|\leq M|r-r_k|\max\Bigl(\frac1r,\frac1{r_k}\Bigr),
\quad r\in(0,1),$$
hence, assuming $(1+\eps)r_k<1$, we obtain
$$v'^2r^2\geq\frac{c^2}2,\qquad
 r\in[(1+\eps)^{-1}r_k,(1+\eps)r_k],\qquad
 \eps=\frac{c^2}{2M}.$$ 
Choosing an infinite subsequence, if necessary, of $\{r_k\}$
so that the intervals above are disjoint and contained in $(0,1)$, it follows that
$$
\begin{aligned}
 \int_0^1 v'^2r\,dr 
 &\geq \sum_{k=1}^\infty \int_{(1+\eps)^{-1}r_k}^{(1+\eps)r_k}v'^2r\,dr \\
 &\geq \sum_{k=1}^\infty \int_{(1+\eps)^{-1}r_k}^{(1+\eps)r_k}\frac{c^2}{2r}\,dr 
 = c^2 \sum_{k=1}^\infty \ln(1+\eps)=\infty,
\end{aligned}
$$
which contradicts \eqref{estvint}.
Consequently \eqref{rdhr} is true.
\end{proof}

\begin{lemma} \label{lem-v0cp}
Let $p>p_S$, let $w$ be a positive solution of \eqref{eqw},
$v(r)=r^\alpha w(r)$, where $\alpha=2/(p-1)$.
Then there exists 
\begin{equation} \label{h0}
v_0:=\lim_{r\to0}v(r).
\end{equation}
If $v_0\ne0$ then $v_0={\L}$.
\end{lemma}

\begin{proof}
Set 
$$a(\xi):=\frac1{p+1}\xi^{p+1}-\frac{{\L}^{p-1}}2 \xi^2.$$
Then passing to the limit as $\rho\to0$ in \eqref{eqv1int}
and using \eqref{rdhr}
we obtain the existence of $\lim_{r\to0}a(v(r))$.
Therefore, the limit set of $v(r)$ as $r\to0$ is disconnected,
hence there exists $v_0=\lim_{r\to0}v(r)$.

Assume on the contrary $v_0\notin\{0,{\L}\}$.
Since $v'(r)r\to0$ as $r\to0$, 
\eqref{eqv1} guarantees $v''(r)r^2\to c\ne0$ as $r\to0$.
Then integrating we obtain $v'(r)=-(c+o(1))/r$ as $r\to0$,
which yields a contradiction.
\end{proof}

\begin{proof}[Proof of Theorem~\ref{thm-uniq}]

(i) Let us first consider equation \eqref{eqw}.

If $p=p_S$ then the statement follows from Lemmas~\ref{lem-rdhr} and~\ref{lem-v0}.

Assume $p>p_S$. 
We proceed as in the proof of \cite[Theorem 3.1]{SZ94}. 
Set $v(r)=r^\alpha w(r)$, where $\alpha=2/(p-1)$.
Lemmas~\ref{lem-v0}, \ref{lem-rdhr} and \ref{lem-v0cp}
imply $v(r)\to {\L}$ as $r\to0$.
Set
$u(r):=v(r)-{\L}$.
Then $u(r)\to0$ and $ru'(r)\to0$ as $r\to0$ and
\begin{equation} \label{equ}
u''+\frac{n-1-2\alpha}r u'+\frac r2 u'+\frac{2(n-2-\alpha)}{r^2}u
+\frac {f(r)}{r^2}=0,
\end{equation} 
where, for $r$ small,
$$
f(r)=({\L}+v(r))^p-{\L}^p-p{\L}^{p-1}v(r)
 ={\L}^p\sum_{k=2}^\infty\frac{p(p-1)\dots(p-k+1)}{k!}\Bigl(\frac{u(r)}{{\L}}\Bigr)^k.$$
Multiplying \eqref{equ} by $r^2u'$ and integrating from $\rho$ to $R$, $R$ small, we obtain
\begin{equation} \label{equint}
\left\{\quad\begin{aligned}
&\Bigl(\frac{r^2u'^2}2+(n-2-\alpha)u^2\Bigr)\Big|_\rho^R 
 +(n-2-2\alpha)\int_\rho^R ru'^2\,dr\\
&\qquad\qquad +\frac12\int_\rho^R r^3u'^2\,dr
+\int_\rho^R f(r)u'\,dr=0.
\end{aligned}\right.
\end{equation}
Since
$$ \begin{aligned}
 \int_\rho^R f(r)u'\,dr 
&= {\L}^{p+1}\sum_{k=3}^\infty\frac{p(p-1)\dots(p-k+2)}{k!}\Bigl(\frac{u(r)}{{\L}}\Bigr)^k\Big|_\rho^R \\
&= {\L}^{p+1}\sum_{k=3}^\infty\frac{p(p-1)\dots(p-k+2)}{k!}
   \Bigl[\Bigl(\frac{u(R)}{{\L}}\Bigr)^k-\Bigl(\frac{u(\rho)}{{\L}}\Bigr)^k\Bigr],
\end{aligned}
$$   
letting $\rho\to0$ in \eqref{equint} we obtain
$$ R^2u'(R)^2+2(n-2-\alpha)u^2(R)\leq M|u(R)|^3,$$
hence $u(R)=0$ for $R$ small.
Consequently, $u\equiv0$.

(ii) Next consider equation \eqref{eqwb}.
Assume that $w$ is an unbounded positive solution of \eqref{eqwb}.
The same arguments as in the proof of Lemma~\ref{lem-ub}
show that \eqref{boundw} is true, 
and the proof of \cite[Lemma~2.1]{Mi09} shows that $w$  
is bounded for $r\geq1$, hence
\begin{equation} \label{estwMM}
w(r)\leq C(1+r^{-2/(p-1)}),\quad r>0.
\end{equation}
If $p>p_S$ then the assertion follows from 
\cite[Theorem~1.2]{Mi10} or
\cite[Proposition~A.1]{MM04}.
If $p=p_S$ then one can modify the proofs of
Lemmas~\ref{lem-v0}--\ref{lem-v0cp} and the arguments in (i) 
to prove that \eqref{estwMM} guarantees $w=U_*$.
In fact, the function $v(r)=w(r)r^\alpha$ now solves the equation
\begin{equation} \label{eqv3}
v''+\Bigl(\frac{n-1-2\alpha}r-\frac r2\Bigr)v'+\frac1{r^2}(v^p-{\L}^{p-1}v)=0,
\qquad r>0,
\end{equation}
instead of \eqref{eqv1},
and the only nontrivial change is in the proof of Lemma~\ref{lem-rdhr}
for $p=p_S$:
Since \eqref{cRrho} reads now
\begin{equation} \label{cRrhoBack}
 c(R)-c(\rho)=\frac12\int_\rho^Rh'(r)^2r^3\,dr>0,
\end{equation}
we have $c_0<c_\infty$ and the property
$c_0=\min b$ does not yield an immediate contradiction.
However, if $c_0=\min b$ then \eqref{cb} implies $h(r)\to1$ and $rh'(r)\to 0$
as $r\to0$, and \eqref{cRrhoBack}, \eqref{cb} imply
\begin{equation} \label{hR}
 \int_0^Rh'(r)^2r^3\,dr=2(c(R)-c_0)\geq R^2h'(R)^2.
\end{equation}
Consider $R\in(0,1)$. The estimate $|rh'(r)|\leq C_h$ for $r\in(0,1)$
and \eqref{hR} guarantee $|h'(R)|\leq C_h$.
Using a bootstrap argument in \eqref{hR} we obtain $|h'(R)|\leq C_hR^k$,
$k=1,2,\dots$, hence $h'\equiv0$.
\end{proof}

\vskip3mm
{\bf Acknowledgements.}
The author was supported in part by the Slovak Research and Development Agency 
under the contract No. APVV-14-0378 and by VEGA grant 1/0319/15.

\def\by{\relax}    
\def\paper{\relax} 
\def\jour{\textit} 
\def\vol{\textbf}  
\def\yr#1{\rm(#1)} 
\def\pages{\relax} 
\def\book{\textit} 
\def\inbook{In: \textit}
\def\finalinfo{\rm} 

\bibliographystyle{amsplain}

\end{document}